\documentclass[11pt, a4paper,reqno]{amsart}

\usepackage[utf8]{inputenc}
\usepackage{hyperref}
\usepackage{amsmath}
\usepackage[dvipsnames]{xcolor}
\usepackage{amssymb}
\usepackage{amsthm}
\usepackage{dsfont}
\usepackage[shortlabels]{enumitem}
\usepackage[english]{babel}

\calclayout

\theoremstyle{definition}
\newtheorem{theorem}{Theorem}[section]

\newtheorem{definition}[theorem]{Definition}

\newtheorem{remark}[theorem]{Remark}
\newtheorem{lemma}[theorem]{Lemma}
\newtheorem{coro}[theorem]{Corollary}
\numberwithin{equation}{section}

\newcommand{\abs}[1]{\left\lvert#1\right\rvert}

\newcommand{\R}{\mathbb R}
\newcommand{\N}{\mathbb N}

\renewcommand{\epsilon}{\varepsilon}

\newcommand{\A}{\mathcal A}
\newcommand{\dx}{\, \mathrm d}
\newcommand{\ee}{\mathrm{e}}
\newcommand{\id}{\mathrm{Id}}

\newcommand{\fra}{\mathfrak a}
\newcommand{\B}{\mathcal B}

\renewcommand{\phi}{\varphi}

\renewcommand{\L}{\operatorname{L}} 
 
\renewcommand{\H}{\operatorname{H}} 

\begin{document}
\allowdisplaybreaks
\title{Nash's $G$ bound for the Kolmogorov equation}

\date{\today}
\thanks{We thank Cl\'ement Mouhot and Rico Zacher for fruitful discussions. Lukas Niebel is funded by the Deutsche Forschungsgemeinschaft (DFG, German Research Foundation) under Germany's Excellence Strategy EXC 2044 --390685587, Mathematics M\"unster: Dynamics--Geometry--Structure.}

\author{Helge Dietert}
\address[Helge Dietert]{Universit\'e Paris Cit\'e and
  Sorbonne Universit\'e, CNRS\\ IMJ-PRG, F-75006 Paris, France.}
\email{helge.dietert@imj-prg.fr}

\author{Lukas Niebel}
\address[Lukas Niebel]{Institut f\"ur Analysis und Numerik,  Westf\"alische Wilhelms-Universit\"at M\"unster\\
Orl\'eans-Ring 10, 48149 M\"unster, Germany.}
\email{lukas.niebel@uni-muenster.de }

\begin{abstract}
We prove Nash's $G$ bound for the Kolmogorov equation with rough coefficients. 
Our proof is inspired by the treatment of the parabolic problem by Nash~(1958) and Fabes and Stroock~(1986). 
To transfer their ideas to the kinetic setting, we employ critical kinetic trajectories. 
From Nash's $G$ bound, we recover the sharp lower bound on the fundamental solution and thus provide an alternative proof of the Harnack inequality for the Kolmogorov equation. 
\end{abstract}
\maketitle

\section{Introduction}
We are concerned with weak solutions $f \in \L^2_{t,x}((0,T) \times \R^d; \H_v^1(\R^d))$ satisfying the Kolmogorov equation 
\begin{equation} \label{eq:int:kol}
  (\partial_t  + v \cdot \nabla_x) f - \nabla_v \cdot ( \mathfrak a \nabla_v f) = 0
\end{equation}
in the distributional sense, where $T>0$. The matrix-valued function $\mathfrak a = \mathfrak a (t,x,v)\in \L^\infty((0,T) \times \R^{2d}; \R^{d \times d})$ is a rough diffusion coefficient satisfying the ellipticity bounds
\begin{enumerate}
\item[] \begin{equation} \label{eq:int:H1}
	0< \lambda := \inf\limits_{\substack{0 \neq \xi \in \R^d,\\  (t,x,v) \in (0,T) \times \R^{2d}}} \frac{\langle \fra(t,x,v) \xi , \xi \rangle}{\abs{\xi}^2}, \tag{\textbf{H1}}
\end{equation}
\item[] \begin{equation} \label{eq:int:H2}
	\Lambda := \sup\limits_{\substack{0 \neq \xi \in \R^d,\\ (t,x,v) \in (0,T) \times \R^{2d}}} \frac{\abs{ \fra (t,x,v) \xi}^2}{\langle\fra (t,x,v) \xi, \xi \rangle}<\infty.\tag{\textbf{H2}}
\end{equation}
\end{enumerate}
We refer to \cite{auscher_weak_2024,dmnz_critical_2025} for more details on the definition of weak solutions and how to work with them. 

The kinetic De Giorgi--Nash--Moser theory is concerned with proving a priori
H\"older continuity of weak solutions to \eqref{eq:int:kol}.  It is inspired by
the treatment of De Giorgi \cite{de_giorgi_sulla_1957}, Nash
\cite{nash_continuity_1958}, and Moser
\cite{moser_harnack_1964,moser_harnacks_1961,moser_pointwise_1971,moser_correction_1967}
of the elliptic and parabolic problems.  The H\"older continuity was first
obtained in \cite{wang_calpha_2011} and strengthened later by a proof of the
stronger Harnack inequality in the influential paper \cite{golse_harnack_2019}.
We mention only the subsequent improvements
\cite{guerand_log-transform_2022,guerand_quantitative_2022,
  dietert2022regularity,dmnz_critical_2025} and refer to the introduction of
\cite{dmnz_critical_2025} for more details.  The existing literature mainly
follows the ideas of De~Giorgi and Moser.

In this article, we explain how to transfer Nash's ideas to the kinetic setting.
Let us first comment on his treatment of the parabolic problem.  Nash's toolbox
consisted of an $\L^\infty$ bound, a bound on the first spatial moment, and the
so-called $G$ bound, i.e.\ a universal lower bound for the spatial integral of
the logarithm of the fundamental solution.  The $G$ bound prevents the kernel
from being too small on large sets and, combined with the moment control, yields
a quantitative overlap of kernels started from nearby points.  This overlap
gives scale-invariant decay of oscillation, from which local H\"older continuity
of weak solutions follows.

Still in the parabolic setting, Aronson later established two-sided Gaussian
bounds for the fundamental solutions to the parabolic problem with rough
coefficients \cite{MR217444} based on the work of Nash. Moreover, he rigorously
established the existence of the fundamental solution \cite{MR435594} and gave
an alternative proof of the lower bound based on the iteration of Moser's
parabolic Harnack inequality.  Nash's proof was revisited and simplified by
Fabes and Stroock in \cite{MR855753}, where they also demonstrated how to obtain
sharp bounds on the fundamental solution directly from the $G$ bound.  Together
with the representation formula for weak solutions, one may derive Moser's
parabolic Harnack inequality and deduce H\"older continuity as a corollary.

Let us now recall the state-of-the-art in the kinetic setting. 
The existence of the fundamental solution to the Kolmogorov equation \eqref{eq:int:kol} with rough coefficients has been proven in \cite{auscher_fundamental_2024}. 
There exists a measurable kernel $\Gamma(t,x,v,s,y,w)$, defined for $s<t$, which is a fundamental solution to the problem \eqref{eq:int:kol}. 
Weak solutions, in the above sense, are represented by this kernel. 
We refer to \cite{auscher_fundamental_2024} for precise statements.

The fundamental solution satisfies the pointwise upper bound
\begin{equation} \label{eq:upper}
	\Gamma(t,x,v,s,y,w) \le \frac{C_0}{(t-s)^{2 d}} \exp\left(-C_1\left(\frac{|x-y-(t-s) w|^2}{(t-s)^3}+\frac{|v-w|^2}{t-s}\right)\right)
\end{equation}
with $C_0 = C_0(d,\lambda,\Lambda)>0$ and $C_1 = C_1(\lambda,\Lambda)>0$, see \cite[Theorem
4.12]{auscher_fundamental_2024}.  Actually, in \cite{auscher_fundamental_2024}
the estimate was proven only for $\abs{\Gamma}$, but $\Gamma \ge 0$ by the maximum principle
for weak solutions.  We emphasise that the proof in
\cite{auscher_fundamental_2024} of the upper bound relies on the
$\L^2$--$\L^\infty$ estimate proven in the context of kinetic De Giorgi--Nash--Moser
theory, see \cite{pascucci_mosers_2004,golse_harnack_2019,dmnz_critical_2025},
and in turn on the kinetic Sobolev inequality.  Using the Harnack inequality,
\cite[Theorem~1.6]{golse_harnack_2019} or
\cite[Theorem~7.5]{dmnz_critical_2025}, a lower bound on $\Gamma$ in terms of the
profile in \eqref{eq:upper} can be deduced via so-called Harnack chains for
kinetic cylinders, see \cite{MR4181953}.

\bigskip

In this article, we provide an alternative approach to prove the lower bound by following the ideas of Nash. 
We make precise connections between Nash's and Moser's work and give a trajectorial proof which does not rely on a spatial Poincar\'e inequality with Gaussian weight as is used in \cite{nash_continuity_1958,MR855753}.
Our main contribution is the following kinetic version of Nash's $G$ bound. 

\begin{lemma} \label{lem:nashG}
	There exists a constant $C = C(d,\lambda,\Lambda) < \infty$ such that for all $\abs{(x,v)} \le 1$ we have
	\begin{equation} \label{eq:nashG}
		\int_{\R^{2d}} \ee^{-\pi(\abs{y}^2+\abs{w}^2)} \log \Gamma(1,y,w,0,x,v) \dx (y,w) \ge - C.
	\end{equation}
\end{lemma}

Let us lay out our strategy to prove it.  We use critical kinetic trajectories
(Section~\ref{sec:kintraj}) developed in \cite{dmnz_critical_2025}, see also
\cite{guerand_quantitative_2022,niebel2023kinetic,anceschi2024poincare,MR4943511},
to prove an $\L^1$-estimate for the difference of the logarithm of the
fundamental solution and the integral on the left-hand side of \eqref{eq:nashG}.
Together with moment bounds (supplied by the upper bound \eqref{eq:upper}) we
deduce the universal $G$ bound in Section~\ref{sec:nashG}.  \medskip
 
From the kinetic $G$ bound, we infer the following sharp lower bound on the fundamental solution in Section \ref{sec:lower}. 

\begin{theorem} \label{thm:lower}
	There exist positive constants $c_0 = c_0(d,\lambda,\Lambda)$ and $c_1 = c_1(d,\lambda,\Lambda)$ such that
	\begin{align*}
		\frac{c_0}{(t-s)^{2 d}} \exp\left(-c_1\left(\frac{|x-y-(t-s) w|^2}{(t-s)^3}+\frac{|v-w|^2}{t-s}\right)\right) \le \Gamma(t,x,v,s,y,w)
	\end{align*}
	for almost all $s<t$ and $x,y,v,w \in \R^d$. 
\end{theorem}

The representation of weak solutions in terms of the fundamental solution
together with the lower and upper bounds readily implies the kinetic Harnack
inequality.  See \cite[Theorem~1.6]{golse_harnack_2019} or
\cite[Theorem~7.5]{dmnz_critical_2025} for precise statements which we do not
make here.

\section{Critical kinetic trajectories}
\label{sec:kintraj}

We recall the notion of kinetic trajectories and the special subclass of
critical ones, which were introduced and developed in~\cite{dmnz_critical_2025}.

\begin{definition}
  \label{def:kintraj}
  Let $(t_0,x_0,v_0)$ and $(t_1,x_1,v_1) \in \R^{1+2d}$ with $t_0 \neq t_1$.

  \noindent
  A family of kinetic trajectories is a map
  \begin{equation*}
    \gamma = \gamma(r) =\gamma(r;(t_0,x_0,v_0),(t_1,x_1,v_1)) = (\gamma_t(r),\gamma_x(r),\gamma_v(r)) \in \R^{1+2d}
  \end{equation*}
  defined for $r \in [0,1]$, that is
  \begin{itemize}
  \item continuous over $r \in [0,1]$ (and in particular bounded),
  \item differentiable over $r \in (0,1)$, 
  \item with endpoints $\gamma(0) = (t_0,x_0,v_0)$ and $\gamma(1) = (t_1,x_1,v_1)$,
  \item satisfying the constraint (kinetic relation)
    $\dot{\gamma}_x(r) = \dot{\gamma}_t(r) \gamma_v(r)$ for $r \in (0,1)$.
  \end{itemize}
  
  \noindent A family of kinetic trajectories is called \emph{critical} if it
  additionally satisfies
  \begin{equation*}
    \det \left( \nabla_{(t_1,x_1,v_1)} \gamma(r;(t_0,x_0,v_0),(t_1,x_1,v_1)) \right)   \sim \abs{ \frac{\gamma_t(r)-t_0}{t_1-t_0} }^{{2+4d}} \quad \text{ as } r \to 0^+
  \end{equation*}
  and
  \begin{equation}
    \label{eq:criticality}
    \left| \dot \gamma_t(r) \right| \left| \left( \left[\nabla_{(x_1,v_1)} \gamma_{x,v}(r;(t_0,x_0,v_0),(t_1,x_1,v_1))\right]^{-1} \right)_{\cdot;2} \right| \sim |\dot{\gamma}_v(r)|  \quad \text{ as } r \to 0^+.
  \end{equation}
  Note that in practice we take $\dot{\gamma}_t$ constant.
\end{definition}

For our purposes, we only need the outcome of the construction given in~\cite{dmnz_critical_2025}: there exists a family of critical kinetic trajectories, explicitly built from forcings with desynchronised logarithmic oscillations, which interpolate between any two endpoints. 
These kinetic trajectories are of the form
\begin{equation*}
	\gamma(r)
  	= \begin{pmatrix}
    \gamma_t(r) \\
    \gamma_x(r) \\
    \gamma_v(r)
  	\end{pmatrix}
  	:= \begin{pmatrix}
    	t_0+ (t_1-t_0)r \\
    	\A_{t_1-t_0}(r) \begin{pmatrix} x_1 \\ v_1
   		\end{pmatrix} 
      	+ \B_{t_1-t_0}(r) \begin{pmatrix} x_0 \\ v_0 \end{pmatrix} 
    \end{pmatrix}, 
\end{equation*}
with matrices $\A_{t_1-t_0},\B_{t_1-t_0} \colon [0,1] \to \R^{2d \times 2d}$, and
satisfy the following properties:
\begin{enumerate}[itemsep=0.2cm]
  \item[\hypertarget{link:1}{\textbf{(1)}}] The map $\gamma$ forms a family of
        critical kinetic trajectories in the sense of Definition
        \ref{def:kintraj} with constant \(\dot \gamma_t = t_1-t_0\).
\item[\hypertarget{link:2}{\textbf{(2)}}] $\A_{t_1-t_0} \colon [0,1] \to \R^{2d \times 2d}$ satisfies
  \begin{enumerate}[topsep=0.5em,itemsep=0.2cm]
  \item[\hypertarget{link:2a}{\textbf{(a)}}] $\A_{t_1-t_0}(0) = 0$, $\A_{t_1-t_0}(1) = \id_{2d}$,
  \item[\hypertarget{link:2b}{\textbf{(b)}}] $\det \A_{t_1-t_0}(r) = r^{2d}$,
  \item[\hypertarget{link:2c}{\textbf{(c)}}] $\abs{(\A_{t_1-t_0}(r)^{-1})_{i;2}} \lesssim (1+\abs{t_1-t_0}) r^{-\frac{1}{2}}$ for $i=1,2$ and $r \in (0,1]$,
  \end{enumerate}
\item[\hypertarget{link:3}{\textbf{(3)}}] $\B_{t_1-t_0} \colon [0,1] \to \R^{2d \times 2d}$ satisfies
  \begin{enumerate}[topsep=0.5em,itemsep=0.2cm]
  \item[\hypertarget{link:3a}{\textbf{(a)}}] $\B_{t_1-t_0}(0) = \id_{2d}$, $\B_{t_1-t_0}(1) = 0$,
  \item[\hypertarget{link:3b}{\textbf{(b)}}] $\det \B_{t_1-t_0}(r) \gtrsim 1$ near $r = 0$.
  \end{enumerate}
\item[\hypertarget{link:4}{\textbf{(4)}}] The trajectory satisfies the following bounds for some universal constants 
  \begin{equation*}
  \begin{cases}
    \abs{\gamma_x(r)-x_0-r(t_1-t_0)v_0} \lesssim   \left(\abs{x_0}+\abs{x_1}\right) r^{\frac32}+ \abs{t_1-t_0} r^{\frac32} \left( \abs{v_0}+\abs{v_1} \right), \\[2mm]
    \abs{\gamma_v(r)-v_0} \lesssim  \abs{t_1-t_0}^{-1} \left( \abs{x_0}+\abs{x_1}\right)r^{\frac12}+ \left( \abs{v_0}+\abs{v_1} \right)r^{\frac12}  , \\[2mm]
    \abs{\dot \gamma_v(r)} \lesssim  \abs{t_1-t_0}^{-1}  \left(\abs{x_0}+\abs{x_1}\right)r^{-\frac{1}{2}}+  \left( \abs{v_0}+\abs{v_1} \right)r^{-\frac12}.
  \end{cases}
\end{equation*}
\end{enumerate}

These features are exactly those needed for our analysis, in particular the criticality condition~\textbf{(2c)}.

\section{Nash's $G$ bound for the Kolmogorov equation}
\label{sec:nashG}

To prove Nash's $G$ bound, we need first a weak $\L^1$ estimate for the
logarithm of weak supersolutions, which is reminiscent of the one of Moser, but
where the comparison is with a logarithmic mean on the full space with Gaussian
weight.  The estimate is reminiscent of \cite[Theorem 6.1]{dmnz_critical_2025},
and we follow the proof closely.

\begin{lemma} \label{lem:L1poin}
	Let $E \subset \R^{2d}$ be a bounded set. Then, for any positive weak supersolution $f >0$ of the Kolmogorov equation \eqref{eq:int:kol} in $ \left( \frac{1}{4},1\right) \times \R^{2d}$, we have 
	\begin{equation*}
		s\abs{\left\{ (t,x,v) \in \left(\frac{1}{4},\frac{3}{4}\right)  \times E :   \log f(t,x,v) - c(f)   > s \right\} } \le C \left( \frac{1}{\lambda}+\Lambda \right)
	\end{equation*}
	for all $s>0$ with
	\begin{equation*}
		c(f) = \int_{\R^{2d}} \log f(1,y,w)\, \ee^{- \pi (\abs{y}^2+\abs{w}^2)} \dx (y,w)
	\end{equation*}
	and some universal constant $C = C(d,E)>0$. 
\end{lemma}

\begin{proof}
	We may assume that $f \ge \epsilon $ for some $\epsilon>0$ and let $\epsilon \to 0^+$ in the end. 
	In view of \cite[Appendix A]{dmnz_critical_2025}, we may argue by formal calculations.

	The fact that $f$ is a positive supersolution to \eqref{eq:int:kol} implies that $g = \log f$ is a supersolution to the Kolmogorov equation with a good quadratic nonlinearity
	\begin{equation}\label{eq:evolution-log}
		(\partial_t + v \cdot \nabla_x) \log f  \ge  \nabla_v \cdot (\fra  \nabla_v \log f) + \langle \fra  \nabla_v \log f , \nabla_v \log f\rangle,
	\end{equation}
	which we can use to absorb any appearing gradient term.
	
	Moreover, we introduce
	\begin{equation*}
		\varphi(y,w) := \ee^{- \frac{\pi}{2} (\abs{y}^2+\abs{w}^2)}
	\end{equation*}
	and note that we have $\abs{\nabla \varphi(y,w)} \lesssim \ee^{- \frac{\pi}{4} (\abs{y}^2+\abs{w}^2)} = \varphi(y,w)^{\frac{1}{2}} \le 1$.

	We prove an $\L^1$ estimate on the positive part of $\log f(t,x,v) - c(f)$ for $(t,x,v) \in (\frac{1}{4},\frac{3}{4}) \times E$. 
	Given $(y,w) \in \R^{2d}$ we choose a critical kinetic trajectory $\gamma$ connecting $(t,x,v)$ with $(1,y,w)$ as in Section~\ref{sec:kintraj}. 
	
	\medskip
	
	 We write
	\begin{align*}
		I&:=\log f(t,x,v) - \int_{\R^{2d}} \log f(1,y,w) \ee^{- \pi (\abs{y}^2+\abs{w}^2)} \dx (y,w)\\
		&=  \int_{\R^{2d}} \Big( \log f(t,x,v) - \log f(1,y,w) \Big) \varphi^2(y,w) \dx (y,w) \\
		&= -  \int_{\R^{2d}} \int_0^{1}\frac{\dx}{\dx r}\log f(\gamma(r))\dx r \ \varphi^2(y,w) \dx (y,w) \\
		&= -  \int_{\R^{2d}}\int_0^{1} \Big( \dot{\gamma}_t(r) [(\partial_t +v\cdot \nabla_x) \log f](\gamma(r)) + \dot{\gamma}_v(r) \cdot [\nabla_v \log f](\gamma(r)) \dx r \Big) \varphi^2(y,w) \dx (y,w)\\
		&\le  - (1-t) \int_{\R^{2d}}\int_0^{1}  [\nabla_v \cdot (\fra \nabla_v \log f)](\gamma(r))  \dx r \ \varphi^2(y,w)   \dx (y,w)\\
		&\hphantom{\le}\,  - (1-t) \int_{\R^{2d}}\int_0^{1}  \langle \fra \nabla_v \log f, \nabla_v \log f \rangle(\gamma(r))  \dx r \ \varphi^2(y,w)  \dx (y,w)\\
		&\hphantom{\le}\,  -  \int_{\R^{2d}}\int_0^{1} \dot{\gamma}_v(r) \cdot [\nabla_v \log f](\gamma(r)) \dx r \ \varphi^2(y,w) \dx (y,w),
	\end{align*}
	where we used the supersolution property~\eqref{eq:evolution-log} in the
    last inequality.

	Next, we employ the change of variables
	\begin{equation*}
		\begin{pmatrix}
			\tilde{y} \\\tilde{w}
		\end{pmatrix} = \Phi(y,w) = \Phi_{r,t,x,v}(y,w) = \begin{pmatrix}
			\gamma_x(r) \\
			\gamma_v(r)
		\end{pmatrix}  = \A_{1-t}(r) \begin{pmatrix}
		y \\ w
		\end{pmatrix} + b(r,t,x,v),
	\end{equation*}
	where $b \colon [0,1] \times \left( \frac{1}{4},\frac{3}{4} \right) \times \R^{2d} \to \R^{2d}$. 
	Next, we do an integration by parts of the first term in the above inequality with respect to $\tilde{w}$. Fix $r \in (0,1)$, then 
	\begin{align*}
		&-\int_{\R^{2d}} [\nabla_v \cdot (\fra \nabla_v \log f)](\gamma(r)) \varphi^2(y,w) \dx (y,w) \\
		&= -\int_{\R^{2d}} [\nabla_v \cdot (\fra \nabla_v \log f)](\gamma_t(r),\tilde{y},\tilde{w})    \varphi^2(\Phi^{-1}(\tilde{y},\tilde{w})) \abs{\det{\A_{1-t}(r)}}^{-1} \dx (\tilde{y},\tilde{w})  \\
		&=  \int_{\R^{2d}} \left\langle [\fra \nabla_v \log f](\gamma_t(r),\tilde{y},\tilde{w}), \nabla_{\tilde{w}}[ \varphi^2(\Phi^{-1}(\tilde{y},\tilde{w}))] \right\rangle \abs{\det{\A_{1-t}(r)}}^{-1} \dx (\tilde{y},\tilde{w}) \\
		&=  2\int_{\R^{2d}} \left\langle[\fra \nabla_v \log f](\gamma_t(r),\tilde{y},\tilde{w}),[\nabla \varphi]^T(\Phi^{-1}(\tilde{y},\tilde{w}))(\A_{1-t}(r)^{-1})_{\cdot; 2}\right\rangle \varphi(\Phi^{-1}(\tilde{y},\tilde{w})) \\
		&\hspace{10.8cm}\cdot \abs{\det{\A_{1-t}(r)}}^{-1} \dx (\tilde{y},\tilde{w}) \\
		&=  2\int_{\R^{2d}} \left\langle[\fra \nabla_v \log f](\gamma(r)),[\nabla \varphi]^T(y,w)(\A_{1-t}(r)^{-1})_{\cdot; 2}\right\rangle \varphi(y,w) \dx (y,w) \\
		&\le 2c_{2c} r^{-\frac{1}{2}}\sqrt{\Lambda} \int_{\R^{2d}} \abs{\nabla_v \log f}_\fra(\gamma(r))\,  \varphi^{\frac{3}{2}}(y,w) \dx (y,w),
	\end{align*}
	where we write $\abs{\cdot}_\fra^2 = \langle \fra \cdot, \cdot \rangle$.  The constant $c_{2c}$ is
    such that $\abs{(\A_{1-t}(r)^{-1})_{\cdot; 2}} \le c_{2c}r^{-\frac{1}{2}}$ for all
    $r \in (0,1]$, see Section~\ref{sec:kintraj} property
    \hyperlink{link:2}{\textbf{(2)}} \hyperlink{link:2c}{\textbf{(c)}}.  Note
    that the boundary term in the integration by parts vanishes due to the decay
    of $\varphi$.

	From here, we obtain for the positive part
	\begin{align*}
		&I_+\le  \frac{1}{2} \int_0^1 \int_{\R^{2d}} \left(\frac{4c_{2c} \sqrt{\Lambda}}{r^\frac{1}{2}} \abs{\nabla_v \log f}_\fra(\gamma(r)) \varphi^{\frac{3}{2}}(y,w)-   \abs{\nabla_v \log f}_\fra^2(\gamma(r))  \varphi^2(y,w) \right)_+ \hspace{-0.3cm} \dx (y,w) \dx r \\ 
		&\hphantom{\le} +\frac{1}{2} \int_0^1\int_{\R^{2d}} \left( \frac{8c_{4}c_5}{\sqrt{\lambda}\,r^{\frac{1}{2} }}\abs{\nabla_v \log f}_\fra(\gamma(r))  \varphi^{\frac{3}{2}}(y,w) -  \abs{\nabla_v \log f}^2_\fra(\gamma(r)) \varphi^2(y,w) \right)_+ \hspace{-0.3cm}\dx (y,w) \dx r.
	\end{align*}
	Here, $c_{4}>0$ is such that $\abs{\dot{\gamma}_v(r)} \le c_{4} (\abs{y}+\abs{w}) r^{-\frac{1}{2}}$
    for all $r \in (0,1]$, see Section~\ref{sec:kintraj} property
    \hyperlink{link:4}{\textbf{(4)}} and we used that
    $\abs{\dot{\gamma}_t} \ge \frac{1}{4}$ uniformly for any choice
    $t \in \left(\frac{1}{4},\frac{3}{4}\right)$.  Moreover, we estimated
    $(\abs{y}+\abs{w})\varphi^2 \le c_5 \varphi^{\frac{3}{2}}$ for some $c_0>0$.
 
	To conclude the proof of the desired estimate, we integrate on $\left(\frac{1}{4},\frac{3}{4}\right) \times E$ and then estimate two terms of the form 
	\begin{align*}
		&\int_{\frac14}^{\frac34} \int_E \int_0^1 \int_{\R^{2d}} \left(Mr^{-\frac{1}{2}}\abs{\nabla_v \log f}_\fra(\gamma(r)) \varphi^\frac{1}{2}(y,w)-   \abs{\nabla_v \log f}_\fra^2(\gamma(r)) \varphi(y,w)\right)_+ \\
		&\hspace{11cm} \varphi(y,w) \dx (y,w) \dx r \dx(x,v) \dx t
	\end{align*}
	for $M \in \left\{4c_{2c}  \sqrt{\Lambda},\frac{8c_{4}}{\sqrt{\lambda} } \right\}$.

	Next, we split up the $r$-integral. 
	First, we treat the case of small $r \in (0,r_0)$ with $r_0 \in (0,1)$ given by property \hyperlink{link:3}{\textbf{(3)}} \hyperlink{link:3b}{\textbf{(b)}} in Section~\ref{sec:kintraj}. 
	We substitute $\tilde{t} = \gamma_t(r) = t+r(1-t)$ and then 
	\begin{equation*}
		(\tilde{x},\tilde{v}) = \Psi(x,v) = \Psi_{r,\tilde{t},y,w}(x,v) = \begin{pmatrix}
		\gamma_x(r) \\
		\gamma_v(r)
		\end{pmatrix}  = \B(r) \begin{pmatrix}
		x \\ v
		\end{pmatrix} + b(r,\tilde{t},y,w)
	\end{equation*}
	for some $\R^{2d}$-valued function $b$ with the noted dependencies. 
	Note that $\Psi(E) \subset \tilde{B} := B_{c(\abs{y}+\abs{w})}(0)$ for some constant $c>0$ by property \hyperlink{link:4}{\textbf{(4)}}.

	The change of variables together with Fubini gives
	\begin{align*}
		&\int_{\frac14}^{\frac34} \int_E \int_0^{r_0} \int_{\R^{2d}} \left(Mr^{-\frac{1}{2}}\abs{\nabla_v \log f}_\fra(\gamma(r))\varphi^\frac{1}{2}(y,w)-   \abs{\nabla_v \log f}_\fra^2(\gamma(r)) \varphi(y,w)\right)_+   \\
		& \hspace{9cm} \varphi(y,w) \dx (y,w) \dx r \dx(x,v) \dx t  \\
		&\le  \int_0^{r_0} \int_{\R^{2d}} \int_{\frac{1}{4}+\frac{3}{4}r}^{\frac{3}{4}+\frac{1}{4}r} \int_{\Psi(E)} \left(Mr^{-\frac{1}{2}}\abs{\nabla_v \log f}_\fra(\tilde{t},\tilde{x},\tilde{v}) \varphi^{\frac{1}{2}}(y,w)-   \abs{\nabla_v \log f}_\fra^2(\tilde{t},\tilde{x},\tilde{v}) \varphi(y,w)\right)_+ \\
		& \hspace{7cm} \frac{1}{1-r} \abs{\det \B(r)}^{-1} \dx(\tilde{x},\tilde{v}) \dx \tilde{t} \,  \varphi(y,w)\dx (y,w) \dx r   \\
		&\le \int_{\R^{2d}} \int_{\frac{1}{4}}^{1} \int_{\tilde{B}} \int_0^{r_0} \left(Mr^{-\frac{1}{2}}\abs{\nabla_v \log f}_\fra(\tilde{t},\tilde{x},\tilde{v}) \varphi^\frac{1}{2}(y,w)-   \abs{\nabla_v \log f}_\fra^2(\tilde{t},\tilde{x},\tilde{v}) \varphi(y,w)\right)_+ \\
		& \hspace{7cm} \sup_{s \in (0,r_0)} \frac{\abs{\det \B(s)}^{-1}}{1-s}  \dx r\dx(\tilde{x},\tilde{v}) \dx \tilde{t} \,  \varphi(y,w)\dx (y,w) .
	\end{align*}
	We introduce $p := p(\tilde{t},\tilde{x},\tilde{v},y,w) := \abs{\nabla_v \log f}_\fra(\tilde{t},\tilde{x},\tilde{v}) \varphi^{\frac12}(y,w)$, which does not depend on $r$. 
	With this notation, we can estimate the inner integral as
	\begin{equation*}
		\int_{0}^{r_0}  \left(  Mr^{-\frac{1}{2}} p-p^2 \right)_+ \dx r  = 2\theta^{\frac{1}{2}}Mp-\theta p^2 \le M^2\le C(d)\left( \frac{1}{\lambda}+\Lambda \right),
	\end{equation*}
	where $\theta = \min \{r_0,M^2/p^2 \}$. 
	We have proven the universal bound for the first term as the $\varphi$ weight allows us to integrate the indicator function of $\tilde{B}$.

	For the integral for large $r \in (r_0,1)$, we employ Young's inequality as
	\begin{align*}
		&\int_{\frac{1}{4}}^{\frac{3}{4}} \int_{E} \int_{r_0}^1 \int_{\R^{2d}} \left(Mr^{-\frac{1}{2}}\abs{\nabla_v \log f}_\fra(\gamma(r)) \varphi^\frac{1}{2}(y,w)-   \abs{\nabla_v \log f}_\fra^2(\gamma(r)) \varphi(y,w)\right)_+ \\
		&\hspace{11cm} \varphi \dx (y,w) \dx r \dx(x,v) \dx t \\
		&\le \int_{\frac{1}{4}}^{\frac{3}{4}} \int_{E} \int_{r_0}^1 \int_{\R^{2d}} \frac{1}{4}M^2r^{-1} \varphi  \dx (y,w) \dx r \dx(x,v) \dx t \\
      &\le C(d,\abs{E})M^2 \le C(d,\abs{E})\left( \frac{1}{\lambda}+\Lambda \right).\qedhere
	\end{align*}
\end{proof}

We now prove Nash's $G$ bound.

\begin{proof}[Proof of Lemma \ref{lem:nashG}]
	Fix $(x,v) \in \R^{2d}$ with $\lvert(x,v)\rvert\le 1$ and set
	\[
	f(t,y,w):=\Gamma(t,y,w,0,x,v),\qquad t>0, \, (y,w) \in \R^{2d}.
	\]
	Then, $f$ is a nonnegative weak solution of the Kolmogorov equation
	\[
		(\partial_t+w\cdot\nabla_y)f-\nabla_w\cdot(\mathfrak a\,\nabla_w f)=0,
	\]
	on $(\frac{1}{4},1)\times\R^{2d}$. 

	Furthermore, we recall that the mass conservation   
	\begin{equation} \label{eq:unitmass}
		\int_{\R^{2d}} f(t,y,w) \dx(y,w) = 1
	\end{equation}
	for all $t>0$ was proven in \cite[Theorem 4.14~(ii)]{auscher_fundamental_2024}. 
	We may replace $f$ by $f+\epsilon$ to ensure positivity and let $\epsilon \to 0$ in the end. 
	In particular, we may apply Lemma~\ref{lem:L1poin} on $(\frac14,1)\times\R^{2d}$.

 	We use the upper bound \eqref{eq:upper} and deduce that
	\begin{equation*}
		M(t) = \int_{\R^{2d}} \abs{(y,w)} f(t,y,w) \dx (y,w) \le C_0
	\end{equation*}
	for all $t \in [\frac{1}{4},\frac{3}{4}]$ and some dimensional constants $c_0,C_0$ depending also on the ellipticity constants. We deduce
	\begin{equation*}
		\int_{\abs{(y,w)} \ge 2C_0} f(t,y,w) \dx (y,w) \le \frac{1}{2C_0} \int_{\abs{(y,w)} \ge 2C_0} \abs{(y,w)} f(t,y,w) \dx (y,w) \le \frac{1}{2}.
	\end{equation*}
	
	From \eqref{eq:unitmass} we deduce
	\begin{equation*}
		\int_{\abs{(y,w)} \le 2C_0} f(t,y,w) \dx(y,w)\ge \frac{1}{2}
	\end{equation*}
	for all $t \in [\frac{1}{4},\frac{3}{4}]$. 
	In particular, 
	\begin{equation*}
		\int_{\frac{1}{4}}^{\frac{3}{4}}\int_{\abs{(y,w)} \le 2C_0} f(t,y,w) \dx(y,w)\ge \frac{1}{4}
	\end{equation*}
	
	The upper bound \eqref{eq:upper} also gives an $\L^\infty$ bound on $Q:=\left[\frac{1}{4},\frac{3}{4}\right] \times B_{2C_0}(0) $, i.e.\ $\abs{f(t,y,w)} \le C_1$ for some $C_1=C_1(d,\lambda,\Lambda)$ and any $(t,y,w) \in Q$. Let $\eta>0$, and $\Omega = \{ (t,y,w) \in Q : f(t,y,w) > \eta \}$, then
	\begin{align*}
		\frac{1}{4} &\le \int_{Q} f(t,y,w) \dx (t,y,w) = \int_{\Omega \cap Q} f(t,y,w) \dx (t,y,w) + \int_{\Omega^c \cap Q} f(t,y,w) \dx (t,y,w) \\
		&\le C_1 \abs{\Omega} + \eta \abs{Q}.
	\end{align*}
	For $\eta$ small enough, we can absorb the last term and obtain $\abs{\Omega} \ge C_2$ for some constant $C_2$ depending on all the other previous constants.
	
	Now, we apply Lemma~\ref{lem:L1poin}, the $\L^1$ estimate for the logarithmic level sets on $Q$ (note $\Omega \subset Q$ and that $Q$ is strictly separated from $\{ 1\} \times \R^{d}$) to obtain
	\begin{equation} \label{eq:par:loglevel}
		\abs{\{ (t,y,w) \in \Omega : \log f(t,y,w) -G(1) > s \} }\le C_3s^{-1}, \; s>0,
	\end{equation}
	where 
	\begin{equation*}
		G(1) = \int_{\R^{2d}} \log f(1,y,w) \ee^{-\pi (\abs{y}^2+\abs{w}^2)} \dx (y,w)
	\end{equation*}
	and $C_3>0$ is a universal constant depending on $d,\lambda, \Lambda$.
	
	For a constant $S$ large enough, we thus obtain that $\abs{\Omega_S}  \ge \frac{\abs{\Omega}}{2}$, where 
	\begin{equation*}
		\Omega_S:=\{ (t,y,w) \in \Omega : \log f(t,y,w) -G(1) \le S \}.
	\end{equation*}
    Together with the good bound, we deduce
	\begin{align*}
		\frac{\abs{\Omega}}{2} \log \eta &\le  \abs{\Omega_S} \log \eta \le \int_{\Omega_S}\log \eta \dx(y,w)  \le   \int_{\Omega_S} \log f (t,y,w)   \dx (y,w) \\
		&\le   \int_{\Omega_S} \big(G(1)+S\big) \dx (y,w) \le (G(1)+S)\abs{\Omega_S}
	\end{align*}
	from which we deduce Nash's $G$ bound as
	\begin{equation*}
		G(1) \ge \frac{\abs{\Omega}}{2\abs{\Omega_S}}\log \eta - S.
        \qedhere
	\end{equation*} 
\end{proof}

\begin{remark}
  This also makes precise the analogy of Moser's weak $\L^1$ estimate for the logarithm of supersolutions and Nash's $G$ bound.  
  Moser's estimate, together with a bound on the first moments and the $\L^\infty$ bound, yields Nash's $G$ bound.  
  We emphasise, however, that the original formulation of Moser's estimate is written for a weight $\varphi$ which has compact support.  
  
  The good set, i.e.\ a set of positive measure where the function is bounded away from zero, is also key in De~Giorgi's approach.  In Moser's argumentation, this is somewhat hidden in the lemma of Bombieri--Giusti
  \cite{dmnz_critical_2025}.
\end{remark}

\begin{remark}
	One could prove the moment bound also directly by an application of the kinetic Nash inequality. We refer to \cite{dmnz_critical_2025} and \cite{hlw_nash_2024}.
\end{remark}

We need the following similar statement, where the integration is in the last variables of $\Gamma$. 
This is immediate in the parabolic setting due to symmetry, but slightly more complicated in the kinetic setting. 

\begin{coro} \label{cor:nashG2}
	There exists a constant $C = C(d,\lambda,\Lambda) < \infty$ such that for all $\abs{(x,v)} \le 1$ we have
	\begin{equation} \label{eq:nashG2} 
		\int_{\R^{2d}} \ee^{-\pi(\abs{y}^2+\abs{w}^2)} \log \Gamma(2,x,v,1,y,w) \dx (y,w) \ge - C.
	\end{equation}
\end{coro}

\begin{proof}
	Let $t>0$ and $(x,v) \in \R^{2d}$. We consider the shifted diffusion coefficient $\fra_t(\cdot,\cdot) = \fra(t-\cdot,\cdot)$ and the fundamental solution $\Gamma^\sharp_{\fra_t}(\sigma,y,w,0,x,v)$ to the Kolmogorov equation 
	\begin{equation*}
		(\partial_\sigma - y \cdot \nabla_w) f-\nabla_w \cdot (\fra_t \nabla_w f) = 0. 
	\end{equation*}
	(This equation is equivalent to \eqref{eq:int:kol} via a coordinate change
    in the position variable.) Using the defining properties and uniqueness of
    the fundamental solution, \cite[Definition~3.3 and
    Lemma~3.4]{auscher_fundamental_2024}, and the observation that
    $(s,y,w) \mapsto \Gamma(t,x,v,s,y,w)$ is the fundamental solution to the adjoint
    problem (\cite[Theorem~4.14~(ii)]{auscher_fundamental_2024}) we deduce
    \begin{equation*}
		\Gamma(t,x,v,t-\sigma,y,w) = \Gamma^\sharp_{\fra_t} (\sigma,y,w,0,x,v).
	\end{equation*}
	Setting $t = 2$, $\sigma = 1$, and applying Lemma~\ref{lem:nashG} to $\Gamma^\sharp_{\fra_t}$ (note that the universal constant depends only on the ellipticity bounds and dimension) gives the desired bound.
\end{proof}

\section{The lower bound on the fundamental solution}
\label{sec:lower}

In this section, we follow \cite[Lemma~2.6 and Theorem~2.7]{MR855753} closely.
We first derive a near-diagonal lower bound and then conclude the sharp lower bound. 
We recall some important facts on scaling and Galilean invariance for the Kolmogorov equation with rough diffusion coefficient. 
For $  (t_1,x_1,v_1), (t_2,x_2,v_2) \in \R^{1+2d}$ we define
\begin{equation*}
	(t_1,x_1,v_1) \circ (t_2,x_2,v_2) := (t_1+t_2,x_1+x_2+t_2v_1,v_1+v_2).
\end{equation*}
The pair $(\R^{1+2d},\circ)$ is a non-commutative group. 
Every $(t,x,v)$ admits a right inverse calculated as $(t,x,v)^{-1} = (-t,-x+tv,-v)$. 
For $r>0$, we set the scaling relation
\begin{equation*}
  [\delta_r f](t,x,v) = f(r^2t,r^3x,rv).
\end{equation*}

 The Kolmogorov equation \eqref{eq:int:kol} is {structurally invariant} with respect to the translation along $\circ$ and the scaling $\delta_r$. 
 To be more precise if $f$ is a weak solution to~\eqref{eq:int:kol}, then $(t,x,v) \mapsto f((t_0,x_0,v_0)\circ (t,x,v))$ and $(t,x,v) \mapsto [\delta_r f](t,x,v)$ are also weak solutions to~\eqref{eq:int:kol} with respective diffusion matrices $\mathfrak a ((t_0,x_0,v_0) \circ \cdot)$ and $\delta_r \mathfrak a$, which both satisfy \eqref{eq:int:H1}--\eqref{eq:int:H2} with the same constants $\lambda$ and $\Lambda$.

\begin{lemma}\label{lem:neardiag}
	There exists $c_0=c_0(d,\lambda,\Lambda)>0$ and a universal $\rho_0\in(0,1]$ such that
	\[
		\Gamma(t,x,v,s,y,w) \ge  \frac{c_0}{(t-s)^{2d}}
	\]
	whenever $t>s$ and for any $(x,v),(y,w) \in \R^{2d}$ with
	\[
	|v-w|\le \rho_0\sqrt{t-s},\qquad|x-y-(t-s)w|\le \rho_0(t-s)^{3/2}.
	\]
\end{lemma}

\begin{proof}
  By scaling and Galilean invariance, it suffices to prove the bound for
  $(s,y,w) = (0,0,0)$, $t = 2$, $\abs{x} \le \rho_0$ and $\abs{v} \le \rho_0$. Note that
  this rescaling will change $\fra$ and hence $\Gamma$ but not the ellipticity
  bounds, which are the only quantities appearing in the estimates. Let $\varphi(\xi,\eta)=\ee^{-\pi(|\xi|^2+|\eta|^2)}$, so $\int_{\R^{2d}}\varphi \dx(\xi,\eta)=1$ and $\varphi \le 1$.

  We fix such $(x,v)$ and then write via the semigroup-property (proven in
  \cite[Theorem~4.14~(i)]{auscher_fundamental_2024}) at the midpoint equal to
  $1$:
	\begin{align*}
		\Gamma(2, x, v,0,0,0)&=\int_{\R^{2d}}\Gamma(2, x, v,1,\xi,\eta)\,\Gamma(1,\xi,\eta,0,0,0)\dx (\xi,\eta)\\
		&\ge \int_{\R^{2d}}\Gamma(2, x, v,1,\xi,\eta)\,\Gamma(1,\xi,\eta,0,0,0) \varphi(\xi,\eta)\dx (\xi,\eta)
	\end{align*}
	By Jensen's inequality for the logarithm and the measure $\varphi(\xi,\eta) \dx (\xi,\eta)$ we deduce
	\begin{align*}
		\log\Gamma(2, x, v,0,0,0)
		&\ge \int_{\R^{2d}}\log\Gamma(2, x, v,1,\xi,\eta)\varphi(\xi,\eta)\dx (\xi,\eta)\\
		&\quad+\int_{\R^{2d}}\log\Gamma(1,\xi,\eta,0,0,0)\varphi(\xi,\eta)\dx (\xi,\eta).
	\end{align*}
	The lower bound is now an immediate consequence of Corollary \ref{cor:nashG2} (as $\abs{(x,v)} \le 1$) for the first term and Lemma \ref{lem:nashG} for the second term.
\end{proof}

We iterate this bound to obtain the sharp decay. 

\begin{proof}[Proof of Theorem \ref{thm:lower}]
We first reduce to a normalised statement. Left-translate by $(s,y,w)^{-1}$ so that $(s,y,w)\mapsto (0,0,0)$ and $(t,x,v)\mapsto (\tau,X,V)$ with
\begin{equation*}
	\tau=t-s,\qquad X:=x-y-\tau w,\qquad V:=v-w.
\end{equation*}
We dilate by $r=\tau^{-1/2}$ so that the time gap becomes~$1$. By the scaling identity of~$\Gamma$,
\begin{equation*}
	\Gamma(t,x,v,s,y,w)=\tau^{-2d}\widetilde\Gamma(1,\bar X,\bar V,0,0,0),
\end{equation*}
where $(\bar X,\bar V)=(X/\tau^{3/2},V/\tau^{1/2})$ and $\widetilde\Gamma$ is the kernel for the scaled field $\delta_\tau \mathfrak a$, still satisfying \eqref{eq:int:H1} and \eqref{eq:int:H2}. 

It suffices to prove
\begin{equation*}
	\widetilde\Gamma(1,\bar X,\bar V,0,0,0)\ \gtrsim\ \exp\big[{-}C_0(|\bar X|^2+|\bar V|^2)\big],
\end{equation*}
with constants depending only on $(d,\lambda,\Lambda)$; this will yield the claimed estimate after undoing the normalisation.

\medskip
Fix $(\bar X,\bar V)\in\R^{2d}$. Choose the smallest integer $k \in \N$ with
\begin{equation*}
	k\ \ge\ k_0\Big( |\bar V|^2+|\bar X|^2\Big),
\end{equation*}
for $k_0=k_0(\lambda,\Lambda)$ to be set large enough below, and partition time by $t_j=j/k$, $j=0,\ldots,k$.

We build reference centres $\{(x_j,v_j)\}_{j=0}^{k}$ in $\R^{2d}$ satisfying
\begin{equation*}
	(x_0,v_0)=(0,0),\qquad (x_k,v_k)=(\bar X,\bar V),
\end{equation*}
\begin{equation} \label{eq:transport-rec}
	x_{j}-x_{j-1}= \Delta t v_{j-1},\; 1\le j\le k,\qquad \Delta t:=1/k,
\end{equation}
and the {small-increment} condition
\begin{equation} \label{eq:vel-incr}
	|v_j-v_{j-1}|  \le \frac{\rho_0}{2}\sqrt{\Delta t}, \; 1\le j\le k
\end{equation}
where $\rho_0$ is the constant from Lemma~\ref{lem:neardiag}. 

One concrete way to produce such a sequence is: take the linear interpolation $v_j^{\rm lin}=\frac{j}{k}\bar V$, and correct it by a quadratic profile $\phi_j$ with $\phi_0=\phi_k=0$ so that $\sum_{j=0}^{k-1} (v_j^{\rm lin}+\phi_j)=\bar X/\Delta t$. 
The increments of a quadratic profile satisfy $|\phi_j-\phi_{j-1}|\lesssim |\bar X|/k^2$. 
Hence for $k$ larger than a dimensional multiple of $|\bar V|^2+|\bar X|^2$, the bound \eqref{eq:vel-incr} holds, while \eqref{eq:transport-rec} then determines $\{x_j\}$ uniquely with $x_0=0$ and $x_k=\bar X$.

\medskip

For a small parameter $\eta\in(0,\rho_0/4]$, define the product set
\begin{equation*}
	\mathcal S_\eta:=\prod_{j=1}^{k-1}\Big(B(x_j,\ \eta\Delta t^{3/2})\times B(v_j,\ \eta\Delta t^{1/2})\Big) \subset (\R^{2d})^{k-1}.
\end{equation*}
Let $((\xi_1,\eta_1),\dots,(\xi_{k-1},\eta_{k-1}))\in\mathcal S_\eta$ and set $(\xi_0,\eta_0)=(0,0)$ and $(\xi_k,\eta_k)=(\bar X,\bar V)$, then for each step $j=1,\dots,k$ the {near-diagonal} hypotheses of Lemma~\ref{lem:neardiag} hold between the points $(t_{j-1},\xi_{j-1},\eta_{j-1})$ and $(t_j,\xi_j,\eta_j)$. Indeed,
\begin{equation*}
	|\eta_j-\eta_{j-1}| \le |\eta_j-v_j|+|v_j-v_{j-1}|+|v_{j-1}-\eta_{j-1}| \le \eta\sqrt{\Delta t}+\frac{\rho_0}{2}\sqrt{\Delta t}+\eta\sqrt{\Delta t} \le \rho_0\sqrt{\Delta t},
\end{equation*}
and, using \eqref{eq:transport-rec},
\begin{align*}
	|\xi_j-\xi_{j-1}-\Delta t\eta_{j-1}| &\le |\xi_j-x_j|+|x_j-x_{j-1}-\Delta tv_{j-1}|+\Delta t|\eta_{j-1}-v_{j-1}| \\
	&\le \eta(\Delta t)^{3/2}+0+\eta(\Delta t)^{3/2} \le \frac{\rho_0}{2}(\Delta t)^{3/2}.
\end{align*}
Therefore, by Lemma~\ref{lem:neardiag},
\begin{equation*}
	\Gamma(t_j,\xi_j,\eta_j,t_{j-1},\xi_{j-1},\eta_{j-1}) \ge\ \frac{c_0}{(\Delta t)^{2d}},\qquad j=1,\dots,k.
\end{equation*}

Applying the semigroup property $(k-1)$-times and restricting the integration to $\mathcal S_\eta$,
\begin{equation*}
	\widetilde\Gamma(1,\bar X,\bar V,0,0,0)
\ge \int_{\mathcal S_\eta}\ \prod_{j=1}^{k}\Gamma(t_j,\xi_j,\eta_j,t_{j-1},\xi_{j-1},\eta_{j-1})
\prod_{j=1}^{k-1} \dx (\xi_j, \eta_j)
\ \ge\ (c_0(\Delta t)^{-2d})^{k}|\mathcal S_\eta|.
\end{equation*}
We have
\begin{equation*}
	|\mathcal S_\eta|=\prod_{j=1}^{k-1} \big(c_d\eta^{2d}(\Delta t)^{2d}\big)
= \big(c_d\eta^{2d}\big)^{k-1}(\Delta t)^{2d(k-1)},
\end{equation*}
where $c_d>0$ is a dimensional constant.
Hence
\begin{equation*}
	\widetilde\Gamma(1,\bar X,\bar V,0,0,0) \ge (\Delta t)^{-2d}\big(c_0 c_d\eta^{2d}\big)^{k}c_d^{-1}\eta^{-2d} \ge\ C_1\big(\alpha\big)^{k},
\end{equation*}
with $C_1=C_1(d,\lambda,\Lambda)>0$ and $\alpha:=c_0 c_d\eta^{2d}\in(0,1)$ for $\eta$ sufficiently small.
Moreover, we dropped the factor $(\Delta t)^{-2d} = k^{2d}$. 
Finally, with $k\asymp |\bar V|^2+|\bar X|^2$, this gives
\begin{equation*}
	\widetilde\Gamma(1,\bar X,\bar V,0,0,0) \ge C_1\exp\big[{-}C_2 (|\bar V|^2+|\bar X|^2)\big],
\end{equation*}
for some $C_2=C_2(d,\lambda,\Lambda)>0$. 
\end{proof}

\bibliographystyle{plain}

\end{document}